\documentclass[12pt]{article}

\usepackage{amsfonts,amsmath,amsthm,amssymb}
\usepackage{epsfig,color}
\usepackage[mathscr]{eucal}

\newtheorem{theorem}{Theorem}

\theoremstyle{definition}

\newtheorem{example}{Example}

\newcommand{\R}{\mathbb{R}}

\newcommand{\fK}{\mathfrak{K}}

\newcommand{\marg}{\vartriangleleft}

\usepackage[T1]{fontenc}

\begin{document}

\title{Multivariate Measures of Concordance for Copulas and their Marginals}         
 
\author{M. D. Taylor}        

\date{\today}          
 
\maketitle

\begin{abstract}
Building upon earlier work in which axioms were formulated for multivariate measures of concordance, we examine properties of such measures.  In particular, we examine the relations between the measure of concordance of an $n$-copula and the measures of concordance of the copula's marginals.
\end{abstract}

\bigskip



\allowdisplaybreaks

\section{Introduction}

Bivariate measures of concordance for an ordered pair of continuous random variables $(X_1,X_2)$ were defined axiomatically by Scarsini in \cite{Scarsini84}.  He showed that such a measure of concordance was a function $\kappa(C)$ of the 2-copula $C$ of $(X_1,X_2)$ and that a number of familiar measures of dependence such as Spearman's rho, Kendall's tau, Gini's coefficient, and Blomqvist's beta were examples of measures of concordance.  Some investigations of bivariate measures of concordance in the spirit of Scarsini's axioms are \cite{Edmiktay03a}, \cite{Edmiktay03b}, \cite{Edwards04}, and \cite{edwards_2008}.  Symmetries of the unit square, $I^2$, played an important role in these investigations.

The idea of a measure of concordance for an ordered $n$-tuple $(X_1,\ldots,X_n)$ of continuous random variables naturally suggested itself.  Examples of generalizations of familiar measures of concordance to a multivariate case can be found in \cite{Joe97}, \cite{nelsen02}, and \cite{Ubeda04}.  Two very similar sets of axioms for multivariate measures of concordance were formulated in \cite{dolati_2006a} and \cite{taylor06}.  We follow here the formulation of \cite{taylor06}, though it is likely that most of the conclusions exhibited hold in both formulations.  The axioms can be formulated in terms of the $n$-copula $C$ of the $n$-tuple $(X_1,\ldots,X_n)$, and symmetries of $I^n$ play an important role.

Some of the consequences of the axioms of \cite{taylor06} were exhibited in \cite{taylor_2008}, though the presentation was less than easily accessible.  This work is devoted mostly to presenting the results \cite{taylor_2008} in a briefer and more readily understandable form.  The most interesting result is that the measure of concordance of an odd-dimensional copula is expressible in terms of the measures of concordance of its even-dimensional marginals.

We are very grateful for the comments and help of our colleagues P. Mikusi\'nski and M. \'Ubeda-Flores in preparing this work.

\section{Copulas and symmetries}

We take $I$ to be the closed unit interval $[0,1]$ and $I^n$ to be the unit $n$-dimensional cube $I \times \cdots \times I$. 

 We will say that a probability measure $\mu$ on the Borel sets of $I^n$ has \emph{uniform one-dimensional marginals} if the following condition holds:  Whenever we have a set of the form $A_1 \times \cdots \times A_n$ and every $A_i$ is $I$ except for possibly some one $A_j$, then
\[
	\mu(A_1 \times \cdots \times A_n) = \lambda(A_j)
\]
where $\lambda$ is one-dimensional Lebesgue measure.

By an $n$-copula, where $n \geq 2$, we mean a function $C : I^n \to I$ associated with a probability measure $\mu$ on $I^n$ with uniform one-dimensional marginals by the equation
\begin{equation} \label{cop_mu}
	C(x_1, \ldots, x_n) = \mu([0,x_1] \times \cdots \times [0,x_n])
\end{equation}
for all $x = (x_1, \ldots, x_n) \in I^n$.  There is a one-to-one correspondence $C \leftrightarrow \mu$ between $n$-copulas and such measures on $I^n$.  

We denote the set of $n$-copulas as Cop($n$).

Here is the connection between copulas and random variables:  If $X_1, \ldots, X_n$ are continuous random variables on a common probability space with respective distribution functions $F_1, \ldots, F_n$, then there is a unique $n$-copula $C$ such that
\[
	F(x_1, \ldots, x_n) = C(F_1(x_1), \ldots, F_n(x_n))
\]
where $x_1, \ldots, x_n \in \R$ and $F$ is the joint distribution function for $(X_1, \ldots, X_n)$.  It turns out that we may almost always assume that each $X_k$ is uniformly distributed over $I$, that our probability space is $I^n$, and we may then identify each $X_k$ with the projection map $X_k: I^n \to I$ defined by $X_k(x_1,\ldots,x_n) = x_k$. 

Two particularly significant $n$-copulas are 
\[
	\Pi^n(x_1, \ldots, x_n) = x_1 \cdots x_n
\]
and
\[
	M^n(x_1, \ldots, x_n) = \min(x_1, \ldots, x_n).
\]
$\Pi^n$ is the copula for $(X_1,\ldots,X_n)$ when the random variables are independent, and $M^n$ is the appropriate copula when each $X_i$ is almost surely a monotone increasing function of every other $X_j$.  See \cite{nelsen_text} and \cite{schweizsklar83} for more information on copulas.

We consider the concept of \emph{marginals} of a copula:  Suppose, for example, that we have a 5-copula $C(x_1,x_2,x_3,x_4,x_5)$.  If we define
\[
	C_2(x_1,x_3,x_4,x_5) = C(x_1,1,x_3,x_4,x_5)
\]
and 
\[
	C_{13}(x_2,x_4,x_5) = C(1,x_2,1,x_4,x_5),
\]
then $C_2$ and $C_{13}$ are marginals of $C$, and it can be seen that $C_2$ is a 4-copula and $C_{13}$ is a 3-copula.  In general, if $C$ is an $n$-copula, then $C_{i_1, \ldots, i_k}$ is the marginal of $C$ obtained by replacing $x_{i_1}, \ldots, x_{i_k}$ with 1.  This marginal will be an $(n-k)$-copula provided $n-k \geq 2$.

If $A$ is a $k$-copula that is a marginal of the copula $C$, we indicate this by writing $A \marg_k C$.

(\textbf{Remark.}  Our notation for marginals is the reverse of what is usually used.  In our example of the marginal $C_2$ of the 5-copula $C$, other authors would tend to use the notation $C_{1345}$ for the marginal.)

By a \emph{symmetry of $I^n$} we understand a one-to-one, onto map  $\phi:I^n \rightarrow I^n$ of the form  
\[ 
  \phi(x_1,\cdots,x_n) = (u_1,\cdots,u_n)  
\] 
where for each $i$  
\[ 
  u_i =  
  \begin{cases} 
    x_{k_i} \text{ or} \\ 
    1-x_{k_i} 
  \end{cases} 
\] 
and where $(k_1,\cdots,k_n)$ is a permutation of $(1,\cdots,n)$. We say that $\phi$ is a \emph{permutation} if for each $i$ we have  $u_i = x_{k_i}$ and is a \emph{reflection} if for each $i$ we have  $u_i = x_i$ or $1-x_i$. 

We define the \emph{elementary reflections}  $\sigma_1, \sigma_2, \cdots, \sigma_n$ by  
\[ 
  \sigma_i(x_1,\cdots,x_n) = (u_1,\cdots,u_n) \text{ where } u_j =  
  \begin{cases} 
    1-x_j \text{ if }j=i \\ 
    x_j \text{ otherwise.} 
  \end{cases} 
\] 
Notice that $\sigma_i \sigma_j = \sigma_j \sigma_i$, that is, the group of reflections is abelian.  By $\sigma^n$ we mean the reflection $\sigma_1\sigma_2\cdots\sigma_n$;  
that is, 
\[ 
  \sigma^n(x_1,\cdots,x_n) = (1-x_1,\cdots,1-x_n). 
\] 
If the choice of $n$ is clear, we shall write $\sigma$ for $\sigma^n$.  

Every symmetry $\xi$ of $I^n$ can be written uniquely in the form $\tau \circ \sigma_{i_1} \circ \cdots \circ \sigma_{i_k}$ (or $\sigma_{j_1} \circ \cdots \circ \sigma_{j_k} \circ \tau'$) where $\tau$ (or $\tau'$) is a permutation and $i_1 < \cdots < i_k$ (or $j_1 < \cdots < j_k$).  We define
\[
	k = | \xi | = \textit{the length of }\xi.
\]

Suppose $C$ is an $n$-copula, $\mu$ is the probability measure associated with $C$ via Equation (\ref{cop_mu}), and $\xi : I^n \to I^n$ is a symmetry of $I^n$.  Then $\xi$ operates on $C$ to produce a new $n$-copula $\xi^*C$ defined by the equation
\begin{equation} \label{symm_cop}
	(\xi^*C)(x) = \mu(\xi([0,x]))
\end{equation}
where $[0,x]$ is the $n$-dimensional rectangle $[0,x_1] \times \cdots \times [0,x_n]$ determined by $x = (x_1, \ldots, x_n) \in I^n$.  It is easily seen that if $\xi$ and $\eta$ are symmetries of $I^n$, then $(\xi \eta)^* = \eta^* \xi^*$.  See \cite{taylor06} and \cite{taylor_2008} for a detailed discussion.

Suppose $C$ is the $n$-copula associated with the $n$-tuple of random variables $(X_1, \ldots, X_n)$.  Then it can be seen that the $n$-copula associated with $(-X_1, X_2, \ldots, X_n)$ is $\sigma_1^*C$.  If each $X_k$ is uniformly distributed over $I$, we can rephrase this idea thus:
\[
	\sigma_1^*C(x_1,x_2, \ldots, x_n) = P(1-X_1 < x_1, X_2 < x_2, \ldots, X_n < x_n).
\]
That is, $\sigma_1^*C$ is the $n$-copula for $(1-X_1,X_2,\ldots,X_n)$.  Similar statements can be made about $\sigma_{i_1}^* \cdots \sigma_{i_k}^* C$.

If $C$ is the $n$-copula associated with the $n$-tuple $(X_1, \ldots, X_n)$, where each $X_k$ is uniformly distributed over $I$, then we define the \emph{survival function} $\overline{C}:I^n \to I$ by
\[
	\overline{C}(x_1, \ldots, x_n) = P(X_1 > x_1, \ldots, X_n > x_n).
\]
Next, given two $n$-copulas $A$ and $B$, we say $B$ is \emph{more concordant} than $A$ and write $A \prec B$ if 
\[
	A \leq B \text{ and } \overline{A} \leq \overline{B}.
\]
It is easily seen that $\overline{A} \leq \overline{B}$ if and only if $\sigma^* A \leq \sigma^* B$.  If $A$ and $B$ are associated with $n$-tuples of random variables, $(X_1, \ldots, X_n)$ and $(Y_1, \ldots, Y_n)$ respectively and $A \prec B$, then we also write $(X_1, \ldots, X_n) \prec (Y_1, \ldots, Y_n)$.

\section{Multivariate measures of concordance}

\subsection{Axioms}

We give a first formulation of the axioms in a way which emphasizes the role of the random variables involved.  This may make their significance a bit clearer.  

By a \emph{measure of concordance} $\kappa$ we mean a function 
that attaches to every $n$-tuple of continuous random variables
$(X_1,\cdots,X_n)$ defined on a common probability space, where $n{\geq}2$,  
a real  number $\kappa(X_1,\cdots,X_n)$ satisfying the following:
\newcounter{naxiom}
\begin{list}{\bfseries A\arabic{naxiom}.}{\usecounter{naxiom}}
  \item  \textbf{(Normalization)} \quad
    $\kappa(X_1,\cdots,X_n)=1$ if each $X_i$ is a.s. an increasing a function 
      of every other $X_j$, and $\kappa(X_1,\cdots,X_n)=0$ if 
      $X_1,\cdots,X_n$ are independent.
  \item  \textbf{(Monotonicity)} \quad
    If $(X_1,\cdots,X_n) \prec (Y_1,\cdots,Y_n)$, then 
        $\kappa(X_1,\cdots,X_n) \leq \kappa(Y_1,\cdots,Y_n)$.
  \item  \textbf{(Continuity)} \quad
    If $F_k$ is the joint distribution function of the random vector 
      $(X_{k1},\cdots,X_{kn})$ and $F$ is the distribution function for 
      $(X_1,\cdots,X_n)$ and $F_k{\rightarrow}F$, then 
      $\kappa(X_{k1},\cdots,X_{kn}) \rightarrow \kappa(X_1,\cdots,X_n)$.
  \item  \textbf{(Permutation Invariance)} \quad
    If $(i_1,\cdots,i_n)$ is a permutation of $(1,\cdots,n)$, then 
      $\kappa(X_{i_1},\cdots,X_{i_n}) = \kappa(X_1,\cdots,X_n)$.
  \item  \textbf{(Duality)} \quad
    $\kappa(-X_1,\cdots,-X_n)=\kappa(X_1,\cdots,X_n)$.
  \item  \textbf{(Reflection Symmetry Property; RSP)} \;
    $\underset{\epsilon_1,\cdots,\epsilon_n}{\sum}
      \kappa(\epsilon_1{X_1},\cdots,\epsilon_n{X_n})=0$ where each 
      $\epsilon_i=\pm{1}$ and the sum is over all possible combinations of $\pm{1}$.
  \item  \textbf{(Transition Property; TP)} \quad
    There exists a sequence of numbers $\{r_n\}$, where $n{\geq}2$, such 
      that for every $n$-tuple of continuous random variables 
      $(X_1,\cdots,X_n)$, we have 
      \[
        r_{n-1} \kappa(X_2,\cdots,X_n) = \kappa(X_1,X_2,\cdots,X_n) + 
	  \kappa(-X_1,X_2,\cdots,X_n).
      \] 
\end{list}

\bigskip

We next restate our axioms in terms of copulas and symmetries of $I^n$.  We say that $\kappa = \{\kappa_n\}$ is a \emph{measure of concordance} if each $\kappa_n$, $n \geq 2$, is a map $\kappa_n : \text{Cop}(n) \to \R$ such that the following holds:

\newcounter{axiom} 
\begin{list}{\bfseries A\arabic{axiom}.}{\usecounter{axiom}} 
  \item  \textbf{(Normalization)} \quad 
    $\kappa_n(M^n)=1$ and $\kappa_n(\Pi^n)=0$. 
  \item  \textbf{(Monotonicity)} \quad 
    If $A \prec B$, then $\kappa_n(A) \leq  
    \kappa_n(B)$. 
  \item  \textbf{(Continuity)} \quad 
    If $C_m \rightarrow C$ uniformly, then  
    $\kappa_n(C_m)\rightarrow \kappa_n(C)$ as $m\rightarrow\infty$. 
  \item  \textbf{(Permutation Invariance)} \quad 
    $\kappa_n(\tau^* C) = \kappa_n(C)$ whenever $\tau$ is a permutation. 
  \item  \textbf{(Duality)} \quad 
    $\kappa_n(\sigma_1^* \cdots \sigma_n^* C)=\kappa_n(C)$. 
  \item  \textbf{(Reflection Symmetry Property; RSP)} \; 
    $\underset{\rho \in R_n}{\sum}\kappa_n(\rho^* C)=0$ where $R_n$ is the group of all reflections of $I^n$. 
  \item  \textbf{(Transition Property; TP)} \quad 
    $r_n \kappa_n(C) = \kappa_{n+1}(E) + \kappa_{n+1}(\sigma_1^* E)$  
    whenever $E$ is an $(n+1)$-copula such that  
    $C(x_1,\cdots,x_n)= E(1,x_1,\cdots,x_n)$. 
\end{list}

\subsection{Examples}

We give some examples of multivariate measures of concordance from \cite{nelsen02} and \cite{taylor06}.  Each of these is a generalization of a well-known bivariate measure of concordance, and the name for each bivariate case has simply been lifted up to the multivariate setting.

Spearman's rho, Gini's coefficient, and Blomqvist's beta all have the form 
\[
	\kappa_n(C) = \alpha_n \; \left( \int_{I^n} (C + \sigma^* C ) \, d\mu_n - \frac{1}{2^{n-1}} \right),
\]
where $\mu_n$ is a probability measure on $I^n$.  Kendall's tau has the form 
\[
	\tau_n(C) = \alpha_n \; \left( \int_{I^n} C \, dC - \frac{1}{2^{n}} \right).
\]
The values of $\alpha_n$ and $r_n$ are displayed in Table \ref{moc_examples}.
\begin{table}
\begin{center}
\begin{tabular}{ l c c }
  \hline
  Measure of concordance	& $\alpha_n$	& $r_n$ \\ \hline
  Spearman's rho	& $\tfrac{(n+1) \, 2^{n-1}}{2^n - (n+1)}$	& $2  \left( \tfrac{n+2}{n+1} \right)  \left( \tfrac{2^n - (n+1)}{2^{n+1} -(n+2)} \right)$ \\
  Gini's coefficient	& $\tfrac{2^n}{2^{n-1}-1}$	& $2  \left( \tfrac{2^{n-1}-1}{2^n -1} \right)$ \\
  Blomqvist's beta	& $\tfrac{2^{n-1}}{2^{n-1}-1}$	& $2  \left( \tfrac{2^{n-1}-1}{2^n -1} \right)$ \\
  Kendall's tau		& $\tfrac{2^n}{2^{n-1}-1}$	& $2  \left( \tfrac{2^{n-1}-1}{2^n -1} \right)$ \\
\end{tabular}
\caption{Parameters for measures of concordance} \label{moc_examples}
\end{center}
\end{table}
For Spearman's rho, $\mu_n = \lambda^n$, $n$-dimensional Lebesgue measure.  For Gini's coefficient, $\mu_n$ is determined by uniformly distributing a mass of $1 / 2^{n-1}$ along each of the one-dimensional diagonals of $I^n$.  A one-dimensional diagonal $D$ is a line segment running from one vertex of $I^n$ to the opposite vertex; as, for example, the diagonal in $I^5$ from $(0,1,0,1,1)$ to $(1,0,1,0,0)$.  For Blomqvist's beta, $\mu_n$ is a unit mass at the point $( \tfrac{1}{2}, \ldots, \tfrac{1}{2})$.

\section{Simple properties}

Here are some properties of measures of concordance from \cite{taylor06}.  Recall that $C_i$ is the marginal of $C$ obtained by setting the $i$th variable of $C$ to $1$.

\begin{theorem} 
  For every measure of concordance $\kappa = \{\kappa_n\}$,  
  the following is true: 
  \begin{list}{}{} 
    \item  (a)  If $C$ is an $n$-copulas  (where $n{\geq}3$) then $r_{n-1}\kappa_{n-1}(C_i) = \kappa_n(C)+\kappa_n(\sigma_i^*C)$ for $i = 1, \ldots, n$. 
    \item  (b)  $r_{n-1} = 1+\kappa_n(\sigma_i^* M^n)$ for $i = 1, \ldots, n$.
    \item  (c)  $r_2=\frac{2}{3}$ and $\kappa_3(\sigma_i^* M^3) = \kappa_3(\sigma_i^* \sigma_j^* M^3) =-\frac{1}{3}$ for $i,j = 1,2,3$ and $i \ne j$. 
     \end{list} 
\end{theorem}

\begin{theorem} 
  For all $n{\geq}2$ and all symmetries  of $I^n$ $\rho$ and $\xi$ such that $|\rho|=|\xi|$   or $|\rho|+|\xi|=n$, we have $\kappa_n(\rho^* M^n)=\kappa_n(\xi^* M^n)$. 
\end{theorem}

In the next result, we refer to a set of axioms given by Scarsini in \cite{Scarsini84} for a bivariate measure of concordance, that is, for a measure of concordance defined for 2-copulas.  These amount to our axioms except that the Reflection Symmetry Property becomes $\kappa_2(\sigma_i^* C) = - \kappa_2(C)$ and the Transition Property is irrelevant.

\begin{theorem} \label{bivarext} 
  Suppose that $\kappa_2:\text{Cop}(2) \rightarrow \R$ is a bivariate measure of concordance in the sense of Scarsini.  For $p=1,2,\cdots$, let us define $\kappa_{p+2}:\text{Cop}(p+2) \rightarrow \R$ and $r_{1+p}$ by  
\[ 
  \kappa_{p+2}(C) = \frac{1}{\binom{p+2}{2}} \,  
    \underset{A \marg_2 C}{\sum} \kappa_2(A) \quad \text{and} 
    \quad r_{1+p} = \frac{2p}{2+p} 
\] 
where it is understood that the summation is over all 2-marginals of $C$.  Then  $\kappa = \{\kappa_n\}$ is a multivariate measure of concordance in our sense. 
\end{theorem}

\section{Reflection-reduction}

Given a measure of concordance $\kappa = \{\kappa_n\}$, we would like to be able to calculate $\kappa_n(\sigma_{i_1}^* \cdots \sigma_{i_k}^* C)$ whenever $C$ is an $n$-copula and $i_1 < \cdots < i_k$.  Notice that beause measures of concordance are invariant under permutations, we can always rearrange the order of the variables in $C$ and calculate $\kappa_n(\sigma_k^* \cdots \sigma_1^* C)$.

In obtaining this result and later ones,  it is convenient to expand our list of transition constants $\{r_n\}_{n=2}^{\infty}$ by setting $r_0 = r_1 = 0$.  It will also be convenient (later) to sometimes talk of 1-copulas and 0-copulas.  We assume the existence of a single 1-copula, namely $\Pi^1(t) = t$ and a single 0-copula, the constant 1.  Notice that if we do this in the presence of a given measure of concordance $\{\kappa_n\}_{n=2}^{\infty}$, then the Transition Property 
\[
	\kappa_n(C) + \kappa_n(\sigma_1^*C) = r_{n-1} \kappa_{n-1}(C_1)
\]
is still true for $n = 1,2$ if we take $\kappa_1$ and $\kappa_0$ to be the zero functions.

The following is proved in a somewhat more general form in \cite{taylor_2008}.

\begin{theorem} \label{ref_red}
	For $C$ an $n$-copula, $n \geq 2$, and $k \leq n$, we have
\begin{equation} \label{ref_red1}
\begin{split}
	\kappa_n(\sigma_k^* \cdots \sigma_1^* C) &= r_{n-1} r_{n-2} \cdots r_{n-k} \, \kappa_{n-k}(C_{1 \ldots k}) \\
	- &r_{n-1} \cdots r_{n-k+1} \, \underset{i_1 < \cdots < i_{k-1} \leq k}{\sum} \kappa_{n-k+1}(C_{i_1 \ldots i_{k-1}}) \\
	+& r_{n-1} \cdots r_{n-k+2} \, \underset{i_1 < \cdots < i_{k-2} \leq k}{\sum} \kappa_{n-k+2}(C_{i_1 \ldots i_{k-2}}) \\
	- & \ldots + (-1)^k \kappa_n(C).
\end{split}	
\end{equation}
\end{theorem}

\begin{proof}
	The result is easily established by induction with respect to $k$.  We consider only the cases $k = 1$ and $k=2$ to exhibit the procedure.

	By the Transition Property, we have $\kappa_n(\sigma_1^* C) = r_{n-1} \kappa_{n-1}(C_1) - \kappa_n(C)$, so the $k=1$ case is trivial.

	Notice that if $i<j$, then $(\sigma_i^*C)_j = \sigma_i^*(C_j)$, so that we might write either of these expressions as $\sigma_i^*C_j$.  Then 
\begin{multline*}
	\kappa_n(\sigma_2^* \sigma_1^* C) = r_{n-1} \kappa_{n-1}(\sigma_1^*C_2) - \kappa_n(\sigma_1^*C) \\
	= r_{n-1} r_{n-2} \kappa_{n-2}(C_{12}) - r_{n-1} \, \Big( \kappa_{n-1}(C_1) + \kappa_{n-1}(C_2) \Big) + \kappa_n(C).
\end{multline*}
Thus we have the $k=2$ case.
\end{proof}

These sorts of calculations can be somewhat simplified by the following notation:  For $n \geq 1$ and $k = 0,1, \ldots, n+1$, set 
\[
	R_{n,k} = 
	\begin{cases}
		1, \quad k=0, \\
		r_n r_{n-1} \cdots r_{n-k+1}, \quad 1 \leq k \leq n-1, \\
		0 \quad k = n, n+1.
	\end{cases}
\]
Then Equation (\ref{ref_red1}) becomes 
\begin{equation} \label{ref_red2}
	\kappa_n(\sigma_k^* \cdots \sigma_1^* C)  =  \sum_{j=0}^k (-1)^{k+j} R_{n-1,j} 
	\underset{i_1 < \cdots < i_{n-j} \leq k}{\sum} \kappa_j (C_{i_1 \ldots i_{n-j}}).
\end{equation}

We use the $R_{n,k}$ to write an extended version of the Transition Property:

\begin{theorem} \label{extended_tp}
	For $C$ an $n$-copula, $n \geq 2$, and $1 \leq k \leq n$, we have
\[
	\sum_{j=0}^k \; \underset{i_1 < \cdots i_j \leq k}{\sum} \kappa_n( \sigma_{i_1}^* \cdots \sigma_{i_j}^* C ) = 
	R_{n-1,k} \; \kappa_{n-k}(C_{12 \ldots k}).
\]
\end{theorem}

\begin{proof}
For $k = 1$, this is just the Transition Property.

The mechanism of the proof is fully displayed by the $k = 2$ case:
\begin{align*}
	\big( \kappa_n(C) &+ \kappa_n(\sigma_2^* C) \big) + \big( \kappa_n(\sigma_1^* C) + \kappa_n(\sigma_2^* \sigma_1^* C) \big) \\ 
	= & \; r_{n-1} \kappa_{n-1}(C_2) + r_{n-1} \kappa_{n-1}(\sigma_1^* C_2) = r_{n-1} r_{n-2} \kappa_{n-2}(C_{12}).
\qedhere
\end{align*}
\end{proof}

\section{Measures of concordance and marginals}

If $A$ is a $p$-copula and $B$ a $q$-copula, we can construct a $(p+q)$-copula $A \otimes B$ by
\[
	A \otimes B (x,y) = A(x) B(y)
\]
where $x \in I^p$ and $y \in I^q$.  This is also true if we take $A$ or $B$ to be $\Pi^1$; see Theorem 6.6.3 of \cite{schweizsklar83}.  We can then show the following:

\begin{theorem} \label{stepdown}
	If $A$ is an $n$-copula, $n \geq 2$, and $k = 1,2,\ldots$, then 
\begin{equation} \label{pi_reduce}
	\kappa_{n+k}(\Pi^k \otimes A) = \frac{1}{2^k} R_{n+k-1,k} \; \kappa_n(A).
\end{equation}
In particular, $\kappa_{n+k}(\Pi^k \otimes M^n) = \tfrac{1}{2^k} R_{n+k-1,k}$.
\end{theorem}

\begin{proof}
	Consider $\Pi^1 \otimes A$.  We may take $\Pi^1$ as the distribution function of a random variable $X$ that is uniformly distributed over $I$ and $A$ as the copula of $(X_1, \ldots, X_n)$ where each $X_i$ is a random variable uniformly distributed over $I$.  If we suppose $X$ and $(X_1, \ldots, X_n)$ are independent, then $\Pi^1 \otimes A$ is the copula of $(X, X_1, \ldots, X_n)$.  It follows that
\[
	\sigma_1^*( \Pi^1 \otimes A) (x,x_1, \ldots, x_n) = P( 1-X < x, \, X_1 < x_1, \ldots, X_n < x_n).
\]
Notice that $1-X$ is uniformly distributed over $I$ and is independent of $(X_1, \ldots, X_n)$.  Thus
\begin{gather*}
	P( 1-X < x, \, X_1 < x_1, \ldots, X_n < x_n) = x \, P( X_1 < x_1, \ldots, X_n < x_n) \\
	= P( X < x, \, X_1 < x_1, \ldots, X_n < x_n) = \Pi^1 \otimes A (x,x_1,\ldots,x_n).
\end{gather*}
That is, $\sigma_1^*( \Pi^1 \otimes A) = \Pi^1 \otimes A$.

Similarly, we have $\sigma_i^* (\Pi^k \otimes A) = \Pi^k \otimes A$ for $i = 1, \ldots, k$.

Next, by the Transition Property,
\[
	\kappa_{n+1}(\Pi^1 \otimes A) + \kappa_{n+1}(\sigma_1^*( \Pi^1 \otimes A )) = r_n \, \kappa_n (( \Pi^1 \otimes A)_1).
\]
Thus
\[
	2 \, \kappa_{n+1}( \Pi^1 \otimes A) = R_{n,1} \, \kappa_n(A)
\]
which is the $k=1$ case of (\ref{pi_reduce}).  Using this $k=1$ case, we see that 
\[
	\kappa_{n+2}( \Pi^2 \otimes A) = \frac{1}{2} \, r_{n+1} \, \kappa_{n+1}( \Pi^1 \otimes A) = 
	\frac{1}{2^2} \, r_{n+1} r_n \, \kappa_n(A).
\]
The proof of the general case is now clear.
\end{proof}

It was brought to our attention by M. \'{U}beda-Flores that if $C$ is a 3-copula, then 
\[
  \kappa_3(C) = \frac{1}{3} \, (\kappa_2(C_1) + \kappa_2(C_2) + \kappa_2(C_3)).
\]
This turns out to be the first in an infinite list of identities in which the measure of concordance of an odd-dimensional copula can always be expressed in terms of the measures of concordance of its even-dimensional marginals.  In general these identities involve the constants $\{r_n\}$; we have a $1/3$ in the identity for the 3-copula because the fact that $r_2 = 2/3$.  We call the identities in this list \emph{\'Ubeda identities} and the constants $a_{n,k}$ \emph{\'Ubeda coefficients}.  Recall here that $A \marg_k C$ means that $A$ is a $k$-copula that is a marginal of the copula $C$.

\begin{theorem} \label{ubeda_thm}
	Let $\kappa = \{\kappa_n\}$ be a measure of concordance.  Then for $m = 1, 2, \ldots$, there exist coefficients $a_{2m+1,2}, a_{2m+1,4}, \ldots, a_{2m+1,2m}$, dependent only on $\{r_n\}$, such that
\begin{equation*} 
  \begin{split}
	\kappa_{2m+1}(C) &= a_{2m+1,2} \underset{A \marg_2 C}{\sum} \kappa_2(A) + a_{2m+1,4} \underset{A \marg_4 C}{\sum} \kappa_4(A) \\ 
	+& \cdots + a_{2m+1,2m} \underset{A \marg_{2m} C}{\sum} \kappa_{2m}(A).
  \end{split}
\end{equation*}
The coefficients may be taken to satisfy
\begin{align*}
	& \frac{1}{2} \, R_{2m,1} = a_{2m+1,2m}, \\
	& \frac{1}{2^3} \, R_{2m,3} = a_{2m+1,2m-2} + \frac{1}{2^2} \, \binom{3}{2} \, R_{2m-1,2} \; a_{2m+1,2m}, \\
	& \cdots \\
	& \frac{1}{2^{2m-1}} \, R_{2m,2m-1} = a_{2m+1,2} + \frac{1}{2^2} \, \binom{2m-1}{2} \, R_{3,2} \; a_{2m+1,4} + \cdots \\ 
	& \hspace{0.5in} + \frac{1}{2^{2m-2}} \, \binom{2m-1}{2m-2} \, R_{2m-1,2m-2} \; a_{2m+1,2m},
\end{align*}
a system of equations that uniquely determines the coefficients.  More concisely, for $p = 0,1, \ldots, m-1$, the coefficients may be taken to satisfy 
\begin{equation} \label{coeff_constraint}
	\frac{1}{2^{2p+1}} \, R_{2m,2p+1} = 
	\underset{\begin{smallmatrix} k+j=p \\ k,j \geq 0   \end{smallmatrix}}{\sum} \frac{1}{2^{2k}} \, \binom{2p+1}{2k} \, R_{2m-1-2j,2k} \; a_{2m+1,2m-2j}.
\end{equation}
\end{theorem}

\begin{proof}[Proof modulo an assumption]
	By Duality and Theorem \ref{ref_red}, 
\begin{align*}
	\kappa_{2m+1}(C) &= \kappa_{2m+1}(\sigma_1^* \cdots \sigma_{2m+1}^* C) \\
	=& - \kappa_{2m+1}(C) + \sum_{j=1}^{2m} (-1)^{j+1} \, R_{2m,j} \, \sum_{A \marg_j C} \kappa_j(A).
\end{align*}
Thus there exist coefficients $b_j$ dependent only on $\{r_n\}$ such that
\[
	\kappa_{2m+1}(C) = \sum_{j=1}^{2m} b_j \, \sum_{A \marg_j C} \kappa_j(A).
\]
Because we can repeatedly apply this trick to $j$-copulas $A$ for which $j$ is odd, there must exist coefficients $a_{2m+1,2i}$ such that
\begin{equation} \label{ubeda_identity}
	\kappa_{2m+1}(C) = \sum_{i=1}^m a_{2m+1,2i} \, \sum_{A \marg_{2i} C} \kappa_{2i}(A).
\end{equation}

Here we make an assumption; we will discuss it after the proof:

We assume there exists a sequence of $2p$-copulas, $\{E^{2p}\}_{p=1}^{\infty}$, having the property that $\kappa_{2p}(E^{2p}) \ne 0$ and all the $(2p-1)$-marginals of $E^{2p}$ are $\Pi^{2p-1}$.  These copulas have the convenient property that if $A \marg_k E^{2p}$ for $k < 2p$, then $A = \Pi^{k}$ and hence $\kappa_k(A) = 0$.

Now choose $p \in \{0,1, \ldots, m-1\}$ and set $C = \Pi^{2p+1} \otimes E^{2m-2p}$ in (\ref{ubeda_identity}).  By Theorem \ref{stepdown},
\[
	\kappa_{2m+1}(C) = \frac{1}{2^{2p+1}} \, R_{2m, 2p+1} \; \kappa_{2m-2p}(E^{2m-2p}).
\]
We want to examine the terms on the other side of (\ref{ubeda_identity}).  Choose $i \in \{1,2,\ldots,m\}$ and suppose that $A \marg_{2i} C$.  If $2i < 2m-2p$, then $\kappa_{2i}(A) = \kappa_{2i}(\Pi^{2i}) = 0$.  So let us suppose that $2i \geq 2m-2p$; this amounts to $i = m-p,m-p+1,\ldots,m$.  Either $A = \Pi^{2i}$ or $A = \Pi^{2i-2m+2p} \otimes E^{2m-2p}$.  In the first case, $\kappa_{2i}(A) = 0$, so we suppose $A$ is of the second form.  By Theorem \ref{stepdown}, we have
\[
	\kappa_{2i}(A) = \frac{1}{2^{2i-2m+2p}} \, R_{2i-1,2i-2m+2p} \; \kappa_{2m-2p}(E^{2m-2p}).
\]
Next notice that the number of ways we can choose $A = \Pi^{2i-2m+2p} \otimes E^{2m-2p}$ as a $2i$-marginal of $C= \Pi^{2p+1} \otimes E^{2m-2p}$ is 
\[
	\binom{2p+1}{2i-2m+2p}.  
\]
It follows that 
\[
	\sum_{A \marg_{2i} C} \kappa_{2i}(A) = \frac{1}{2^{2i-2m+2p}} \, \binom{2p+1}{2i-2m+2p} \, R_{2i-1,2i-2m+2p} \; \kappa_{2m-2p}(E^{2m-2p})
\]
assuming $2i \geq 2m-2p$.  If we substitute these values into (\ref{ubeda_identity}) and divide out $\kappa_{2m-2p}(E^{2m-2p})$, we obtain
\begin{equation*}
	\frac{1}{2^{2p+1}} \, R_{2m, 2p+1} = \sum_{i=m-p}^m \frac{1}{2^{2i-2m+2p}} \, \binom{2p+1}{2i-2m+2p} \, R_{2i-1,2i-2m+2p} \; a_{2m+1,2i}.
\end{equation*}
Making the substitutions $k=i-m+p$ and $j=m-i$ yields (\ref{coeff_constraint}), and we are done.
\end{proof}

What is the status of our assumption in the proof?  To begin with, in ``most'' cases, a suitable $E^{2p}$ does exist.

\begin{example}
	Examples are given in \cite{dolati06} and \cite{ubeda_2008} of $n$-copulas $E^n \ne \Pi^n$ having the property that each of their $(n-1)$-marginals are $\Pi^{n-1}$.  The following particularly simple instance, based on a construction in \cite{dolati06}, will suffice for our purposes:  For $n \geq 2$, define 
\begin{equation} \label{En_def}
	E^n(x_1, \ldots, x_n) = \Pi_{i=1}^n x_i + \theta \, \Pi_{i=1}^n x_i \, (1-x_i)
\end{equation}
where $\theta$ is a parameter such that $-1 \leq \theta \leq 1$.  We see that the boundary conditions for a copula are trivially true for $E^n$.  The density of $E^n$ is 
\[
	\delta_n(x_1,\ldots,x_n) = \frac{\partial^n E^n}{\partial x_1 \ldots \partial x_n} = 
	1 + \theta \, \Pi_{i=1}^n (1-2x_i).
\]
Since $0 \leq x_i \leq 1$, we see that $\delta_n \geq 0$, hence $E^n$ is an $n$-copula.  It is easily checked that every $(n-1)$-marginal of $E^n$ is $\Pi^{n-1}$.

We want to see that $\kappa_{2m}(E^{2m}) \ne 0$ for certain measures of concordance $\kappa$.

Recall that $\sigma = \sigma_1 \cdots \sigma_n$.  The density of $\sigma^* E^n$ must be 
\[
	\delta_n(1-x_1, \ldots, 1-x_n) = 1 + (-1)^n \, \theta \, \Pi_{i=1}^n (1-2x_i).
\]
It follows that 
\[
	\sigma^* E^n(x_1, \ldots, x_n) = \Pi_{i=1}^n x_i + (-1)^n \, \theta \, \Pi_{i=1}^n x_i \, (1-x_i).
\]
Thus 
\[
	E^n + \sigma^* E^n = \begin{cases}
		2 \, \Pi^n, \quad n \text{ odd}, \\
		2 \, E^n, \quad n \text{ even}.
	\end{cases}
\]

Now consider measures of concordance of the form 
\[
	\kappa_n(C) = \alpha_n \, \left( \int_{I^n} (E^n + \sigma^* E^n ) \, d\mu_n - \frac{1}{2^{n-1}} \right)
\]
where $\alpha_n \ne 0$ and $\mu_n$ is a probability measure with all its mass on $(0,1)^n$.  (Spearman's rho, Gini's coefficient, and Blomqvist's beta have this form.)  Since $\kappa_n(\Pi^n) = 0$, we see that 
\[
	\int_{I^n} ( \Pi^n + \sigma^* \Pi^n ) \, d\mu_n = 2 \, \int_{I^n} \Pi^n \, d\mu_n = \frac{1}{2^{n-1}}.
\]
Then
\[
	\kappa_{2m}(E^{2m}) = 2 \, \alpha_{2m} \, \theta \, \int_{I^{2m}} \Pi_{i=1}^{2m} x_i \, (1-x_i) \; d\mu_{2m}(x_1, \ldots, x_{2m}).
\]
Thus $\kappa_{2m}(E^{2m}) \ne 0$ if $\theta \ne 0$.

Therefore the assumption in the proof of Theorem \ref{ubeda_thm} holds for Spearman's rho, Gini's coefficient, and Blomqvist's beta.
\end{example}

\begin{example}
	Kendall's tau is given by
\[
	\tau_n(C) = \alpha_n \, \left( \int_{I^n} C \, dC - \frac{1}{2^n} \right)
\]
where $\alpha_n \ne 0$.  Using the $E^n$ defined by (\ref{En_def}) and the formula for its density, one calculates that 
\[
	\tau_n(E^n) = \alpha_n \, \theta \, \left( \left( \frac{1}{6} \right)^n + \left( - \frac{1}{6} \right)^n \right),
\]
which is nonzero if $\theta \ne 0$ and $n$ is even.  Thus the assumption holds for Kendall's tau.
\end{example}

However the assumption of the existence of $E^{2p}$ in the proof of Theorem \ref{ubeda_thm} is not required for the validity of the theorem.  A general proof is given in \cite{taylor_2008} and a description of the \'Ubeda coefficients equivalent to those satisfying Equation (\ref{coeff_constraint}) of Theorem \ref{ubeda_thm} is derived without the assumption, but the proof is less accessible and the notation is markedly different.

\begin{example}
	It turns out that the $E^{2p}$ copulas assumed in the proof of Theorem \ref{ubeda_thm} do not always exist for a given measure of concordance.

To see this, let $\kappa_2$ be a bivariate measure of concordance and extend it to a measure concordance $\kappa = \{\kappa_n\}_{n=2}^{\infty}$ in the manner described in Theorem \ref{bivarext}.  Notice that we can choose our \'Ubeda coefficients to be 
\[
	a_{2m+1, 2k} = \begin{cases}
			\frac{1}{\binom{2m+1}{2}} \quad \text{for } k=1, \\
			0 \quad \text{for } k \geq 2.
		\end{cases}
\]
But if the $E^{2p}$ copulas exist for this $\kappa$, then the \'Ubeda coefficients must satisfy Equation (\ref{coeff_constraint}), and we must, in particular, have $a_{2m+1,2m} = (1/2) \, R_{2m,1} \ne 0$.  This is a contradiction.

However we happen to know that we could, if we wished, choose the \'Ubeda coefficients so they satisfied Equation (\ref{coeff_constraint}).  We thus also see from this example that the \'Ubeda coefficients are not unique.
\end{example}

Our last comment about the \'Ubeda coefficients is to probe a little more deeply into their structure.  We first define a sequence of numbers $\{\gamma_{2k+1}^*\}_{k=0}^{\infty}$ by the system of equations
\begin{equation} \label{gamma_def}
\begin{split}
	&\binom{1}{0} \, \gamma_1^* = 1, \\
	&\binom{3}{0} \, \gamma_3^* + \binom{3}{2} \, \gamma_1^* = 1, \\
	&\binom{5}{0} \, \gamma_5^* + \binom{5}{2} \, \gamma_3^* + \binom{5}{4} \, \gamma_1^* = 1, \\
	&\text{etc.}
\end{split}
\end{equation}
The first few terms of the sequence are displayed in Table \ref{gamma_table}.  
\begin{table}
\begin{center}
\renewcommand{\arraystretch}{1.25}
\begin{tabular}{ l | c c c c c }  \hline
  $k$	& $1$	& $3$  & $5$  & $7$  & $9$ \\ \hline
  $\gamma_k^*$  & $1$  & $-2$  & $16$  & $-272$  & $7936$ \\ 
  $\frac{1}{2^k} \, \gamma_k^*$	&  $\frac{1}{2}$  & $- \frac{1}{4}$  & $\frac{1}{2}$  & $- \frac{17}{8}$  & $\frac{31}{2}$ \\
 \end{tabular}
\caption{Values of $\gamma_k^*$ and $\gamma_k^*/2^k$} \label{gamma_table}
\end{center}
\end{table}
Next we note that for any measure of concordance $\kappa$ and associated sequence of transition constants $\{r_n\}_{n=2}^{\infty}$ we have 
\begin{equation} \label{R_relation}
	R_{m+n, m} \, R_{n, p} = R_{m+n, m+p}
\end{equation}
for $n = 1,2,\ldots$ and $m,p = 0,1,2,\ldots$.

\begin{theorem}
	If the \'Ubeda coefficients $a_{2m+1,2k}$ of a measure of concordance satisfy Equation (\ref{coeff_constraint}), then
\begin{equation} \label{ubeda_gamma}
	a_{2m+1,2k} = \frac{1}{2^{2m+1-2k}} \, \gamma_{2m+1-2k}^* \, R_{2m, 2m+1-2k}
\end{equation}
for $m = 1,2,\ldots$ and $k = 1, \ldots, m$.
\end{theorem}

\begin{proof}
	Let us write the right-hand side of (\ref{coeff_constraint}) with $a_{2m+1, 2m-2j}$ replaced by $(1/2^{2j+1}) \, \gamma_{2j+1}^* \, R_{2m, 2j+1}$ and simplify using $k+j = p$, (\ref{gamma_def}), and (\ref{R_relation}).  
\begin{gather*}
	\sum_{j=0}^p \frac{1}{2^{2k}} \, \binom{2p+1}{2k} \, R_{2m-1-2j,2k} \, \frac{1}{2^{2j+1}} \, \gamma_{2j+1}^* \, R_{2m, 2j+1} \\
	= \sum_{j=0}^p \frac{1}{2^{2p+1}} \, \binom{2p+1}{2p-2j} \,  \gamma_{2j+1}^* \, R_{2m, 2p+1} \\
	= \frac{1}{2^{2p+1}} \, R_{2m, 2p+1} \, \sum_{j=0}^p  \binom{2p+1}{2p-2j} \,  \gamma_{2j+1}^* \\
	= \frac{1}{2^{2p+1}} \, R_{2m, 2p+1}.
\end{gather*}
Since the \'Ubeda coefficients are uniquely determined by (\ref{coeff_constraint}), we have 
\[
	a_{2m+1,2m-2j} = \frac{1}{2^{2j+1}} \, \gamma_{2j+1}^* \, R_{2m,2j+1}
\]
which amounts to our desired result.
\end{proof}

\begin{example}
	Let us introduce the symbol $\fK_k(C)$ for $\sum_{A \marg_k C} \kappa_k(A)$.  Then consulting Table \ref{gamma_table}, we see that for any measure of concordance $\kappa$ we have, for instance,
\begin{gather*}
	\kappa_9(C) = - \frac{17}{8} \, R_{8,7} \, \fK_2(C) + \frac{1}{2} \, R_{8,5} \, \fK_4(C) \\
	- \frac{1}{4} \, R_{8,3} \, \fK_6(C) + \frac{1}{2} \, R_{8,1} \, \fK_8(C).
\end{gather*}
\end{example}

\section{An asymptotic result}

Suppose that we compute $\kappa_s(C_s)$ for the copula $C_s$ of the random vector $(X_1,\cdots,X_s)$ and then consider a new, enlarged random vector $(X_1,\cdots,X_s,X_{s+1},\cdots,X_{s+p})$ where each new random variable $X_{s+k}$ is a monotone increasing function of, let us say, $X_1$.  If $C_{s+p}$ is the copula of the new random vector, then one is tempted to suspect that $\kappa_{s+p}(C_{s+p}) \rightarrow 1$ as $p \rightarrow \infty$.  However, this is often not the case.  Here is a result to justify that statement:

\begin{theorem}
	Let $\{X_i\}_{i=1}^{\infty}$ be a sequence of continuous random variables such that each $X_i$ is almost surely a monotone increasing function of every other $X_j$.  Suppose that we are given $\{ \epsilon_i \}_{i=1}^{\infty}$ and $i_1 < i_2 < \cdots < i_s$ such that $\epsilon_i = -1$ if $i=$ some $i_k$ and otherwise $\epsilon_i = 1$.  Let $C_n$ be the copula of $(\epsilon_1 X_1, \ldots, \epsilon_n X_n)$.  If $\kappa$ is a measure of concordance such that $r_n \to r$, then $\kappa_n(C_n) \to (r-1)^s$ as $n \to \infty$.
\end{theorem}

\begin{proof}
	For $n>s$, the copula of $(\epsilon_1 X_1, \ldots, \epsilon_n X_n)$ is $C_n = \sigma_{i_1}^* \cdots \sigma_{i_s}^* M^n$.  By Permutation Invariance and Theorem \ref{ref_red}, we have 
\begin{align*}
	\kappa_n(C_n) &\quad = \quad \kappa_n(\sigma_1^* \cdots \sigma_s^* M^n) \\
	=& \quad r_{n-1} \cdots r_{n-s} - \binom{s}{s-1} \, r_{n-1} \cdots r_{n-s+1} + \binom{s}{s-2} \, r_{n-1} \cdots r_{n-s+2} \\ 
	-& \cdots + (-1)^s \quad \to  \quad r^s - \binom{s}{s-1} \, r^{s-1} + \binom{s}{s-2} \, r^{s-2} - \cdots + (-1)^s \\
	= & \quad (r-1)^s.  \qedhere
\end{align*} 
\end{proof}

For Spearman's rho, Gini's coefficient, Blomqvist's beta, and Kendall's tau, we see that $r_n \to 1$ and hence $\kappa_n(C_n) \to 0$.  However for the measure of concordance constructed in Theorem \ref{bivarext}, we have $r_n \to 2$ so that $\kappa_n(C_n) \to 1$.  Thus some multivariate measures of concordance are very sensitive to the presence of any amount of independence and others are not.

\section{Questions}

\begin{enumerate}

\item  Can we give interesting examples of measures of concordance of degree $m$ for every natural number $m$?

We say that a measure of concordance $\kappa$ is of \emph{degree $m$} provided the following hold:
\begin{enumerate}
\item  For every $n$ and for all $n$-copulas $A, B$, the map 
\[
	t \mapsto \kappa_n( (1-t) \, A + t \, B) 
\]
is a polynomial in $t$.
\item  $\sup_n \text{degree} \; \kappa_n( (1-t) \, A + t \, B) = m$.
\end{enumerate}
It is readily seen that Spearman's rho, Gini's coefficient, and Blomqvist's beta are all of degree one and that Kendall's tau is of degree two.  The measure of concordance in Theorem \ref{bivarext} will be of whatever degree the ``seed'' $\kappa_2$ is that is used in its construction.

\item  Can we characterize measures of concordance of degree one?  Of degree $m$?

Bivariate measures of concordance (in the sense of Scarsini) of degree one have been characterized in \cite{edwards_2008}.  It is possible that the constructions of this proof could be extended to the multivariate case, at least for the degree one case, and that degree one measures of concordance would have the form 
\[
	\kappa_n(C) = \alpha_n \, \left( \int_{I^n} (C + \sigma^*C) \, d\mu_n - \frac{1}{2^{n-1}} \right)
\]
for a suitably restricted class of probability measures $\mu_n$.

\item  Can we find the minimum values of measures of concordance of degree one?  Of degree $m$?

\item  Is there a systematic way to find sample versions of measures of concordance?  One suspects there might be a nice answer to this question for measures of concordance of degree one if we possessed a characterization of such measures of concordance.

\end{enumerate}

\noindent Department of Mathematics, University of Central Florida\\
\noindent Orlando, Fl 32816-1364, USA \\
\noindent mtaylor@pegasus.cc.ucf.edu

\end{document}